\documentclass[reqno,a4paper,11pt]{amsart}
\usepackage{amsmath,amssymb}
\usepackage[
 hidelinks,
 pdftitle={Improved bounds in Birch's theorem for forms in many variables},
 pdfauthor={Yijie Diao and Lena Wurzinger}
]{hyperref}
\usepackage[margin=1.35in]{geometry}
\usepackage{array}
\usepackage{booktabs}

\numberwithin{equation}{section}

\newcommand{\ZZ}{\mathbb Z}
\newcommand{\QQ}{\mathbb Q}
\newcommand{\RR}{\mathbb R}
\newcommand{\x}{\mathbf{x}}
\newcommand{\h}{\mathbf h}
\newcommand{\e}{\mathrm e}
\newcommand{\eps}{\varepsilon}
\newcommand{\dint}{\,\mathrm d}

\theoremstyle{plain}
\newtheorem{theorem}{Theorem}[section]
\newtheorem{proposition}[theorem]{Proposition}
\newtheorem{lemma}[theorem]{Lemma}
\newtheorem{corollary}[theorem]{Corollary}

\theoremstyle{definition}
\newtheorem{remark}[theorem]{Remark}

\title{Improved bounds in Birch's theorem for forms in many variables}
\author{Yijie Diao}
\address{ISTA\\
Am Campus 1\\
3400 Klosterneuburg\\
Austria}
\email{yijie.diao@ist.ac.at}
\author{Lena Wurzinger}
\address{ISTA\\
Am Campus 1\\
3400 Klosterneuburg\\
Austria}
\email{lena.wurzinger@ist.ac.at}
\date{\today}

\begin{document}

\begin{abstract}
We improve the best known result on the number of variables needed for the smooth Hasse principle for homogeneous forms of degree \(d\geq5\).
\end{abstract}

\maketitle
\thispagestyle{empty}
\setcounter{tocdepth}{1}
\tableofcontents

\section{Introduction}

Let \(F\in\ZZ[x_1,\ldots,x_n]\) be a homogeneous form of degree \(d\geq3\), write \(\x=(x_1,\ldots,x_n)\), and let \(\sigma\geq0\) denote the affine dimension of the singular locus of the affine hypersurface \(F=0\) over \(\overline{\QQ}\).
We say that \(F\) satisfies the \emph{smooth Hasse principle} if nonsingular zeros over \(\RR\) and over every \(\QQ_p\) imply a nonsingular zero over \(\QQ\).
Birch \cite{birch} proved that this holds whenever
\[
 n-\sigma>(d-1)2^d.
\]

For \(d=3\), Birch's condition is \(n-\sigma\geq17\).
Davenport \cite{davenport_cubic_16} proved that every cubic form in at least \(16\) variables has a nontrivial rational zero, and Heath-Brown \cite{heath_brown_cubic_14} later improved this bound for arbitrary cubic forms to \(14\) variables.
For nonsingular cubic forms, Heath-Brown \cite{heath_brown_cubic_10} proved the Hasse principle in \(10\) variables, and Hooley \cite{hooley_nonary_cubic} later improved the required number to \(9\).
Assuming the Riemann hypothesis for certain Hasse--Weil \(L\)-functions, Hooley \cite{hooley_octonary_cubic} further reduced it to \(8\) variables.

For \(d=4\), Birch's theorem gives \(n-\sigma\geq49\).
Browning and Heath-Brown \cite{browning_heath_brown_quartic} reduced this to \(41\).
Hanselmann \cite{hanselmann_quartic} further improved it to \(40\).
Marmon and Vishe \cite[Theorem~1.1]{marmon_vishe_quartic} proved the smooth Hasse principle for quartic forms whenever \(n-\sigma\geq30\).

For \(d\geq5\), Browning and Prendiville \cite[Theorems~1.1 and~1.2]{bp} proved the smooth Hasse principle whenever
\[
 \begin{aligned}
 n-\sigma&\geq \frac34d2^d-2d + 1,
 &&\text{for }5\leq d\leq9,\\
 n-\sigma&\geq\left(d-\frac12\sqrt d\right)2^d,
 &&\text{for }d\geq10.
 \end{aligned}
\]

There are also several related results that develop Birch's method in other directions.
Brandes \cite{brandes_linear_spaces} used the circle method to count linear spaces on intersections of hypersurfaces.
For systems of forms, Dietmann \cite{dietmann_weyl_systems} and, independently, Schindler \cite{schindler_weyl_variant} replaced Birch's condition on \(V^*\) by a sharper condition involving the singular loci of the nonzero forms in the rational pencil.
Schindler \cite{schindler_higher_expansions}, extending work of Vaughan and Wooley \cite{vaughan_wooley_higher_expansions} on Waring's problem, obtained further terms in the asymptotic formula for general forms under stronger hypotheses.
Yamagishi \cite{yamagishi_hessian} refined Birch's original condition by replacing \(\sigma\) with a Hessian invariant \(\mathcal H_F\leq\sigma\).
Lee \cite{lee_function_fields} proved an analogue of Birch's theorem for forms over function fields.
Very recently, Hase-Liu \cite{hase_liu_function_fields} reduced the number of variables needed in Birch's theorem over function fields from exponential to quadratic for smooth hypersurfaces under suitable conditions.

In this paper, we prove sharper uniform conditions in every degree \(d\geq5\).
For \(3\leq k<d\), define
\begin{equation}\label{eq:intro-C}
 C_{d,k}
 =2^d\left\{
 d+\frac38-\frac12\left(k+\frac dk\right)
 +\left(\frac dk-1\right)(2^{-k}-2^{1-d})
 \right\}-1.
\end{equation}
Put \(C_d=\max_{3\leq k<d}C_{d,k}\).

\begin{theorem}\label{thm:main}
Let \(d\geq5\), and let \(F\in\ZZ[x_1,\ldots,x_n]\) be a form of degree \(d\).
Write \(\sigma\) for the affine dimension of the singular locus of the affine hypersurface \(F=0\) over \(\overline{\QQ}\).
Suppose that
\[
 n-\sigma\geq
 \frac34d2^d-2d-\left(3\cdot2^{d-3}-3\right),
 \qquad\text{for }d=5,6;
\]
or
\[
 n-\sigma>C_d,
 \qquad\text{for }d\geq7.
\]
Then the smooth Hasse principle holds for \(F\).
\end{theorem}

For \(5\leq d\leq12\), the following table compares the least integer values of \(n-\sigma\) for which the smooth Hasse principle holds as a consequence of \cite{bp} and of Theorem~\ref{thm:main}.
\begin{center}
\begin{tabular}{@{}c|rrrrrrrr@{}}
\(d\)&5&6&7&8&9&10&11&12\\ \hline
\cite{bp}
&111&277&659&1{}521&3{}439&8{}621&19{}132&42{}058\\
Theorem~\ref{thm:main}
&101&255&621&1{}468&3{}388&7{}675&17{}146&37{}882
\end{tabular}
\end{center}

\medskip

We also give an explicit estimate for \(C_d\).

\begin{corollary}\label{cor:asymptotic}
As \(d\to\infty\),
\[
 C_d=
 \left(d-\sqrt d+O(1)\right)2^d.
\]
Moreover, for every \(d\geq7\), we have
\[
 C_d<\left(d-\sqrt d+1\right)2^d.
\]
\end{corollary}

For \(d=5,6\), our result requires \(3\cdot2^{d-3}-2\) fewer variables than Browning and Prendiville \cite{bp}.
For large \(d\), we obtain an extra saving of \((\frac12\sqrt d+O(1))2^d\).

\medskip
We improve the result in three ways.
First, Lemma~\ref{lem:moduli} sharpens the count of admissible moduli.
Second, Lemma~\ref{lem:average} uses a Gaussian average to carry out the first van der Corput step without selecting a single first shift; this follows the method of Marmon and Vishe \cite[Section~3.2]{marmon_vishe_quartic}.
Third, Proposition~\ref{prop:cubic-family} applies the final cubic estimate uniformly to the resulting shifts.

\subsection*{Notation}
In a dyadic decomposition, \(x\sim X\) means \(X\leq x<2X\).
We write
\[
 \e(z)=\exp(2\pi iz),\qquad
 \|\theta\|=\min_{a\in\ZZ}|\theta-a|,\qquad
 |\mathbf v|=\max_i|v_i|.
\]
We also write \((x)_+=\max\{x,0\}\).
Finally, \(p^j\Vert q\) means that \(p^j\mid q\) but \(p^{j+1}\nmid q\).
\subsection*{Acknowledgments}

We are grateful to our advisor Tim Browning for his comments throughout this work.
We also thank Christian Bernert and Julia Brandes for helpful discussions.

\section{Auxiliary results}\label{sec:inputs}

We introduce the three estimates that will be used in the proof of Theorem~\ref{thm:main}.

\subsection{Counting moduli}

Following \cite[Eq.~(3.12)]{bp}, for \(q\geq1\), we write \(q=b_qc_q^2d_q\), where
\[
 b_q=\prod_{\substack{p^r\Vert q\\r\leq2}}p^r,\qquad
 c_q=\prod_{\substack{p^r\Vert q\\r>2}}p^{\lfloor r/2\rfloor},\qquad
 d_q=\prod_{\substack{p^r\Vert q\\r>2,\ 2\nmid r}}p.
\]
Using this factorization, define \(\eta_q\) as in \cite[Eq.~(3.24)]{bp}:
\begin{equation}\label{eq:eta}
 \eta_q=\max\left\{
 (b_q^3c_qd_q^2)^{1/(2^d-2)},
 (b_q^3c_q^2d_q)^{1/(5\cdot2^{d-2}-4)}
 \right\}.
\end{equation}

Browning and Prendiville \cite[Lemma~4.4]{bp} prove the following bound.
\begin{equation}\label{eq:bp-modulus-count}
 \#\{q \in \mathbb{N}:\eta_q\leq R\}\ll_d
 R^{7\cdot2^{d-4}-\frac54}.
\end{equation}
The next lemma lowers the exponent in \eqref{eq:bp-modulus-count} by coupling the two height inequalities on each dyadic box.

\begin{lemma}\label{lem:moduli}
For \(R \geq 1\), we have
\[
 \#\{q \in \mathbb{N} :\eta_q\leq R\}\ll_d
 R^{A_d}(\log(2R))^4,
\]
where \(A_d=3\cdot2^{d-3}-1\).
\end{lemma}

\begin{proof}
Let
\[
 e_q=\prod_{p^5\Vert q}p,
 \qquad
 s=\frac{c_q}{d_qe_q}.
\]
If \(p^r\Vert q\) with \(r>2\), then the exponent of \(p\) in \(s\) is zero for \(r=3,5\) and is at least two otherwise.
Thus \(s\) is squareful, while \(d_q\) and \(e_q\) are squarefree and \(e_q\mid d_q\).
Moreover,
\[
 c_q=d_qe_qs,
 \qquad
 q=b_qd_q^3e_q^2s^2,
\]
so the tuple \((b_q,d_q,e_q,s)\) uniquely determines \(q\).

Partition these tuples into dyadic boxes.
Write \(b_q\sim B_0\), \(d_q\sim C_0\), \(e_q\sim E_0\), and \(s\sim S_0\).
There are \(O(B_0)\) choices for \(b_q\) and there are \(O(S_0^{1/2})\) squareful integers in \([S_0,2S_0)\).
Since \(e_q\mid d_q\), write \(d_q=e_qf\).
For fixed \(e_q\), the condition \(d_q\sim C_0\) leaves \(O(C_0/e_q+1)\) choices for \(f\).
A nonempty box has \(E_0\ll C_0\), and hence
\[
 \sum_{e_q\sim E_0}
 \left(\frac{C_0}{e_q}+1\right)
 \ll C_0.
\]
For an upper bound, we may count all \(d_q\) and \(e_q\) in their dyadic intervals with \(e_q\mid d_q\).
Thus the number of tuples in the box satisfies
\begin{equation}\label{eq:moduli-box-count}
 \ll B_0C_0S_0^{1/2}.
\end{equation}

Now impose the condition \(\eta_q\leq R\).
Equation~\eqref{eq:eta} yields
\begin{align}
 b_q^3c_qd_q^2
 &\leq R^{2^d-2},
 \label{eq:moduli-height-first}\\
 b_q^3c_q^2d_q
 &\leq R^{5\cdot2^{d-2}-4}.
 \label{eq:moduli-height-second}
\end{align}
After substituting \(c_q=d_qe_qs\) and passing to the dyadic ranges, these inequalities become
\begin{align}
 (B_0C_0)^3S_0E_0
 &\ll R^{2^d-2},
 \label{eq:moduli-dyadic-height-first}\\
 (B_0C_0)^3S_0^2E_0^2
 &\ll R^{5\cdot2^{d-2}-4}.
 \label{eq:moduli-dyadic-height-second}
\end{align}
Multiplying \eqref{eq:moduli-dyadic-height-first} and \eqref{eq:moduli-dyadic-height-second} and taking sixth roots, we find
\begin{equation}\label{eq:moduli-coupled-box-bound}
 B_0C_0S_0^{1/2}E_0^{1/2}
 \ll R^{(2^d-2+5\cdot2^{d-2}-4)/6}
 =R^{3\cdot2^{d-3}-1}
 =R^{A_d}.
\end{equation}
Therefore, \eqref{eq:moduli-box-count} and \eqref{eq:moduli-coupled-box-bound} show that every nonempty box contains \(O(R^{A_d})\) tuples.
Moreover, \eqref{eq:moduli-height-first} and \eqref{eq:moduli-height-second} bound each dyadic parameter by a fixed power of \(R\).
Thus there are \(O_d((\log(2R))^4)\) boxes, and summing over them proves the lemma.
\end{proof}

\begin{remark}
The exponent in Lemma~\ref{lem:moduli} is essentially optimal.
Indeed, consider
\[
 q=uw^4,
 \qquad
 (u,w)=1,
 \qquad
 \mu^2(u)=\mu^2(w)=1,
\]
where
\[
 u\leq R^{2^{d-2}}
 \quad\text{and}\quad
 w\leq R^{2^{d-3}-1}.
\]
Coprime squarefree pairs have positive density, so there are \(\gg_d R^{A_d}\) such pairs.
For each of them, \(b_q=u\), \(c_q=w^2\), and \(d_q=1\).
Hence
\[
 u^3w^2\leq R^{2^d-2},
 \qquad
 u^3w^4\leq R^{5\cdot2^{d-2}-4},
\]
and therefore \(\eta_q\leq R\).
\end{remark}

\subsection{Averaging the first difference}\label{sec:averaging}

If a nonsingular local zero is missing, there is nothing to prove.
We therefore assume from now on that \(F\) has a nonsingular zero over \(\RR\) and over every \(\QQ_p\).

Let \(B\geq1\) be a scaling parameter.
Fix a nonsingular real zero \(\x_0\), and use homogeneity to rescale it so that \(|\x_0|<1\).
After permuting the coordinates, assume that \(\partial_1F(\x_0)\neq0\).
Following \cite[Section~5]{bp}, choose \(\rho>0\) sufficiently small and set
\[
 \omega(\x)=w\bigl(\rho^{-1}\|\x-\x_0\|_2\bigr),
 \qquad
 w(t)=
 \begin{cases}
  \exp\bigl(-1/(1-t^2)\bigr),& |t|<1,\\
  0,& |t|\geq1.
 \end{cases}
\]
The weight \(\omega\) is nonnegative and smooth, with compact support and bounded derivatives of every order.
Since \(\partial_1F(\x_0)\neq0\), continuity allows \(\rho\) to be chosen so that, for some \(c_0>0\),
\[
 \min_{\mathbf y\in\operatorname{supp}(\omega)}
 |\partial_1F(\mathbf y)|\geq c_0.
\]
By homogeneity, there is a constant \(c_1>0\), depending only on \(F\) and \(\omega\), such that
\begin{equation}\label{eq:lower-bound-partial-F}
 |\partial_1F(\x)|\geq c_1B^{d-1},
 \qquad\text{for }\x/B\in\operatorname{supp}(\omega).
\end{equation}
All translated products of weights below then belong to one fixed smooth class.
Write
\begin{equation}\label{eq:exponential-sum}
 S_\omega(\alpha;B)=\sum_{\x\in\ZZ^n}\omega(\x/B)
 \e(\alpha F(\x)).
\end{equation}
For \(\h\in\ZZ^n\), define the first difference of \(F\) by
\begin{equation}\label{eq:first-difference}
 \Delta_\h F(\x)=F(\x+\h)-F(\x).
\end{equation}
The associated first difference sum is
\begin{equation}\label{eq:first-difference-sum}
 S_{\h}(\alpha;B)=\sum_{\x\in\ZZ^n}
 \omega((\x+\h)/B)\omega(\x/B)
 \e(\alpha\Delta_\h F(\x)).
\end{equation}

The following estimate extends the averaged van der Corput estimate of Marmon and Vishe \cite[Lemma~3.4]{marmon_vishe_quartic} from quartic forms to forms of degree \(d\).

\begin{lemma}\label{lem:average}
Let \(\eps>0\), \(B\geq2\), \(1\leq H\leq B\), and \(0<t\leq1\).
Let \(a\in\ZZ\) and \(q\in\mathbb N\) with \((a,q)=1\), and put \(\lambda=(HB^{d-1})^{-1}\).
Define
\begin{equation}\label{eq:shift-average}
 \Sigma_H=\max_{(t-B^\eps\lambda)_+\leq |z|
 \leq2t+B^\eps\lambda}
 \sum_{|\h|\leq HB^\eps}|S_{\h}(a/q+z;B)|.
\end{equation}
Then, for every fixed \(J\geq1\), we have
\begin{equation}\label{eq:average}
 \int_{t\leq|z|\leq2t}|S_\omega(a/q+z;B)|\dint z
 \ll B^{-J}+B^\eps(t+\lambda)
 \left(\frac BH\right)^{(n-1)/2}\Sigma_H^{1/2}.
\end{equation}
The implied constant depends only on \(F,\omega,\eps\), and \(J\).
\end{lemma}

\begin{proof}
Let \(J\) be fixed.
Cover \(t\leq|z|\leq2t\) by \(O(1+t/\lambda)\) intervals \(I(z_0)\) of radius \(\lambda\), choosing every center \(z_0\) in the set \(\{z:t\leq|z|\leq2t\}\).
For each center \(z_0\), use \eqref{eq:exponential-sum} to put
\[
 \mathcal M(z_0)=\int_{\RR}
 \exp(-(z-z_0)^2/\lambda^2)
 |S_\omega(a/q+z;B)|^2\dint z.
\]
The Gaussian is \(\gg 1\) on \(I(z_0)\), so Cauchy--Schwarz gives
\begin{equation}\label{eq:interval-cauchy}
 \int_{I(z_0)}|S_\omega(a/q+z;B)|\dint z
 \ll\lambda^{1/2}\mathcal M(z_0)^{1/2}.
\end{equation}

We first record the van der Corput inequality for lattice sums in the form needed here.
If \(\mathcal H\subset\ZZ^n\) is finite and
\[
 \#\{\mathbf y\in\ZZ^n:f(\mathbf y+\mathbf u)\neq0
 \text{ for some }\mathbf u\in\mathcal H\}\ll B^n,
\]
then translation averaging followed by Cauchy--Schwarz gives
\begin{equation}\label{eq:general-vdc}
 \begin{aligned}
 (\#\mathcal H)^2\left|\sum_{\x\in\ZZ^n}f(\x)\right|^2
 &\ll B^n\sum_{\mathbf y\in\ZZ^n}
 \left|\sum_{\mathbf u\in\mathcal H}f(\mathbf y+\mathbf u)\right|^2\\
 &=B^n\sum_{\h\in\ZZ^n}N(\h)
 \sum_{\mathbf y\in\ZZ^n}
 f(\mathbf y+\h)\overline{f(\mathbf y)}.
 \end{aligned}
\end{equation}
where
\[
 N(\h)=\#\{(\mathbf u,\mathbf v)\in\mathcal H^2: \mathbf u-\mathbf v=\h\}.
\]

Choose \(\delta>0\) sufficiently small, depending only on \(F\) and \(\omega\).
Apply \eqref{eq:general-vdc} to
\[
 f_z(\x)=\omega(\x/B)\e((a/q+z)F(\x)),
\]
with
\[
\mathcal H=
 \bigl([1,\delta B]\times[1,H]^{n-1}\bigr)\cap\ZZ^n.
\]
Since \(\omega\) has fixed compact support and \(|\mathbf u|\ll B\) for \(\mathbf u\in\mathcal H\), this support condition holds.
After enlarging the implied constant we may assume that \(B\) is large enough that \(\#\mathcal H\asymp BH^{n-1}\); the bounded remaining values of \(B\) are harmless.
Moreover, \(N(\h)\leq\#\mathcal H\), and the correlation in \eqref{eq:general-vdc} is exactly \(S_\h(a/q+z;B)\).
The difference set is contained in \([ -\delta B,\delta B]\times[-H,H]^{n-1}\).
Multiply \eqref{eq:general-vdc} by the Gaussian and integrate in \(z\).
Using the triangle inequality and the bound on \(N(\h)\), we obtain
\begin{equation}\label{eq:gaussian-vdc}
 \mathcal M(z_0)\ll
 \left(\frac BH\right)^{n-1}\!
 \sum_{\substack{|h_1|\leq\delta B\\
                  |h_j|\leq H\\
                  2\leq j\leq n}}
 \left|\int_{\RR}\exp(-(z-z_0)^2/\lambda^2)
 S_{\h}(a/q+z;B)\dint z\right|.
\end{equation}
For \(u\in\RR\), we have the identity
\begin{equation*}
 \int_{\RR}\exp(-(z-z_0)^2/\lambda^2)
 \e((a/q+z)u)\dint z
 =\lambda\sqrt\pi
 \e((a/q+z_0)u)
 \exp(-\pi^2\lambda^2u^2).
\end{equation*}
Apply this with \(u=\Delta_\h F(\x)\) after expanding \(S_{\h}\) in \eqref{eq:first-difference-sum}.
The summand vanishes unless \(\omega(\x/B)\omega((\x+\h)/B)\neq0\), and on this support Taylor's formula applied to \eqref{eq:first-difference} gives
\begin{equation}\label{eq:first-difference-taylor}
 \Delta_\h F(\x)
 =h_1\partial_1F(\x)+O(HB^{d-1} + (h_1^2+H^2)B^{d-2}).
\end{equation}
For every shift in \eqref{eq:gaussian-vdc} with \(|h_1|>HB^\eps\), we have
\begin{equation*}
 HB^{d-1}+(h_1^2+H^2)B^{d-2}
 \ll(\delta+B^{-\eps})|h_1|B^{d-1}.
\end{equation*}
Choose \(\delta\) sufficiently small in terms of the constant in \eqref{eq:first-difference-taylor} and the lower bound \eqref{eq:lower-bound-partial-F}.
For all sufficiently large \(B\), the error term is then at most half the main term, and hence
\[
 |\Delta_\h F(\x)|\geq \frac{c_1}{2}|h_1|B^{d-1},
\]
for every such shift.
The remaining bounded values of \(B\) are absorbed into the final implied constant.
Hence the factor in the preceding Fourier transform satisfies
\[
 \exp(-\pi^2\lambda^2(\Delta_\h F(\x))^2)
 \leq\exp(-c\lambda^2|h_1|^2B^{2d-2})
 =\exp(-c|h_1|^2/H^2)
 \leq\exp(-cB^{2\eps}).
\]
Here \(c=\pi^2c_1^2/4>0\).
Choose an auxiliary exponent \(J_0=J_0(J,d,n)\) large enough that the subsequent polynomial losses leave an error \(O(B^{-J})\).
Since the sums over \(\x\) and \(\h\) have polynomial length, this bound shows that the shifts with \(|h_1|>HB^\eps\) contribute \(O_{J_0}(B^{-J_0})\).
The Gaussian weight in \eqref{eq:gaussian-vdc}, together with the trivial estimate \(|S_{\h}(a/q+z;B)|\ll B^n\), similarly shows that the range \(|z-z_0|>B^\eps\lambda\) contributes \(O_{J_0}(B^{-J_0})\).
For the remaining shifts and frequencies, the triangle inequality gives
\begin{equation*}
 \begin{aligned}
 &\sum_{\substack{|h_1|\leq HB^\eps\\
                  |h_j|\leq H\\
                  2\leq j\leq n}}
 \left|\int_{|z-z_0|\leq B^\eps\lambda}
 \exp(-(z-z_0)^2/\lambda^2)
 S_{\h}(a/q+z;B)\dint z\right|\\
 &\qquad\leq
 \int_{|z-z_0|\leq B^\eps\lambda}
 \exp(-(z-z_0)^2/\lambda^2)
 \sum_{|\h|\leq HB^\eps}|S_{\h}(a/q+z;B)|\dint z\\
 &\qquad\ll B^\eps\lambda\Sigma_H.
 \end{aligned}
\end{equation*}
Combining with \eqref{eq:gaussian-vdc} and the two tail estimates gives
\begin{equation}\label{eq:local-average}
 \mathcal M(z_0)
 \ll_{J_0} B^{-J_0}
 +B^\eps\lambda\left(\frac BH\right)^{n-1}\Sigma_H.
\end{equation}
Because every center lies in the original annulus, the truncated frequencies in the preceding display lie in the range used to define \(\Sigma_H\).
Finally, the cover has \(O(1+t/\lambda)\) intervals.
The main term obtained from \eqref{eq:interval-cauchy} and \eqref{eq:local-average} is
\[
 \ll B^\eps(t+\lambda)
 \left(\frac BH\right)^{(n-1)/2}\Sigma_H^{1/2}.
\]
Since \(t\leq1\) and \(\lambda^{-1/2}\leq B^{d/2}\), choosing \(J_0\) sufficiently large makes the summed error \(O(B^{-J})\).
This proves \eqref{eq:average}.
\end{proof}

\subsection{Uniform differencing}

Lemma~\ref{lem:average} produces a sum over the first shifts.
Since the later differencing choices may depend on the first shift, we need to estimate the resulting family uniformly.
The next proposition gives the required estimate.

For \(d\geq5\), define
\begin{equation}\label{eq:tau}
 \tau=2-2^{4-d}.
\end{equation}
For the parameters used below, define
\begin{equation}\label{eq:cubic-bound}
 K_B(\alpha,q,Q)
 =\frac{\sqrt q}{B}
 +\sqrt{\|q\alpha\|BQ^\tau}
 +\frac{Q^{\tau/6}\min\{c_q,d_qQ^\tau\}^{1/6}}
 {b_q^{1/2}(c_qd_q)^{1/3}}.
\end{equation}

\begin{proposition}\label{prop:cubic-family}
Let \(d\geq5\), \(B\geq2\), \(\eps>0\), and \(1\leq Q\leq B\).
Let \(q\) be a positive integer, and suppose that \(\|q\alpha\|=|q\alpha-a|\) for some integer \(a\) with \((a,q)=1\).
Suppose that
\[
 q\leq B^2,\qquad \|q\alpha\|\leq B^{-1},
\]
and \(K_B(\alpha,q,Q)\leq B^\eps Q^{-1}\).
Then, with \(H_1=Q^{1/2^{d-4}}\),
\begin{equation}\label{eq:cubic-family}
 \sum_{|\h|\leq H_1B^\eps}|S_{\h}(\alpha;B)|
 \ll B^{n+O_{d,n}(\eps)}H_1^\sigma.
\end{equation}
The implied constant depends only on \(F,\omega\) and \(\eps\).
\end{proposition}

Fix \(d\geq5\), \(B\geq2\), \(\eps>0\), \(1\leq Q\leq B\), \(\alpha\in\RR\), and a positive integer \(q\leq B^2\).
Put \(r=d-3\), so that \(r\) differences reduce \(F\) to degree at most \(3\).
In the iterated Cauchy--Schwarz, the density introduced at stage \(i\) occurs with exponent \(2^{r-i}\).
For \(1\leq i\leq r\), put \(H_i=Q^{1/2^{r-i}}\), so we have
\begin{equation}
 H_i^{2^{r-i}}
 =Q.
 \label{eq:differencing-scale}
\end{equation}
The common scale makes the powers of the \(H_i\) telescope in the backward density calculation of \cite[Proof of Lemma~3.5]{bp}.
The product of the shift lengths controls the height of the cubic polynomial obtained after all \(r\) differences.
By \eqref{eq:tau}, it satisfies
\begin{equation}
 H_1\cdots H_r
 =Q^\tau.
 \label{eq:differencing-product}
\end{equation}
Given \(\h_1,\ldots,\h_i\), write the polynomial obtained after the first \(i\) differencing steps as
\begin{equation}
 F_{\h_1,\ldots,\h_i}
 =\Delta_{\h_i}\cdots\Delta_{\h_1}F.
 \label{eq:iterated-difference}
\end{equation}
Put \(\omega_\varnothing=\omega\), and use the recursion
\[
 \omega_{\h_1,\ldots,\h_i}(\mathbf y)
 =
 \omega_{\h_1,\ldots,\h_{i-1}}(\mathbf y+\h_i/B)
 \omega_{\h_1,\ldots,\h_{i-1}}(\mathbf y).
\]
The corresponding iterated sum is defined as follows.
\[
 S_{\h_1,\ldots,\h_i}(\alpha;B)
 =
 \sum_{\x\in\ZZ^n}
 \omega_{\h_1,\ldots,\h_i}(\x/B)
 \e\bigl(\alpha F_{\h_1,\ldots,\h_i}(\x)\bigr).
\]
Each \(\omega_{\h_1,\ldots,\h_i}\) is a product of at most \(2^r\) translates of \(\omega\) and contains the untranslated factor \(\omega(\mathbf y)\).
Thus these weights have common compact support and uniform derivative bounds depending only on \(d\) and \(\omega\).

For a form \(G\) over \(K=\QQ\) or \(\mathbb F_p\), define its singular locus over \(K\) by
\[
 \operatorname{Sing}_K(G)
 =
 \{\x\in\mathbb A_K^n:
 \partial_1G(\x)=\cdots=\partial_nG(\x)=0\}.
\]
Write \(\mathbb F_\infty=\QQ\).
For \(\mathbf a\in\ZZ^n\), let \([\mathbf a]_p\) denote its reduction in \(\mathbb F_p^n\), and put \([\mathbf a]_\infty=\mathbf a\).

We record two geometric bounds from \cite{bp} in the notation used below.
Let \(G\in\ZZ[x_1,\ldots,x_n]\) be a form of degree \(e\), let \(s\geq0\), and suppose that \(\nu=\infty\), or that \(\nu=p\) is a prime with \(p\nmid e\).
Put
\[
 B_\nu(G,s)
 =\{\mathbf y\in\mathbb A_{\mathbb F_\nu}^n:
 \dim\operatorname{Sing}_{\mathbb F_\nu}
 (\mathbf y\cdot\nabla G)\geq s\}.
\]
By \cite[Lemma~2.4]{bp}, this is an affine variety defined by \(O_e(1)\) equations of degree \(O_e(1)\), and
\begin{equation}\label{eq:singular-shift-bound}
 \dim B_\nu(G,s)
 \leq n-
 \bigl(s-\dim\operatorname{Sing}_{\mathbb F_\nu}(G)\bigr)_+.
\end{equation}

Let \(\mathcal P\) be a finite set of places among \(\infty\) and the rational primes.
For each \(\nu\in\mathcal P\), let \(X_\nu\subset\mathbb A_{\mathbb F_\nu}^n\) be an affine variety of dimension at most \(k_\nu\), defined by at most \(D\) equations of degree at most \(D\).
The dimension growth bound \cite[Lemma~2.5]{bp} gives, for \(T\geq1\),
\begin{equation}\label{eq:dimension-growth-bound}
 \begin{aligned}
 &\#\{\mathbf a\in\ZZ^n\cap[-T,T]^n:
 [\mathbf a]_\nu\in X_\nu\text{ for every }\nu\in\mathcal P\}\\
 &\qquad\leq A(D,n)^{|\mathcal P|}
 \sum_{\nu\in\mathcal P}T^{k_\nu}
 \prod_{\substack{\mu\in\mathcal P\\k_\mu<k_\nu}}
 \mu^{-(k_\nu-k_\mu)},
 \end{aligned}
\end{equation}
where \(\infty^{-1}\) is interpreted as \(0\).

For a polynomial \(g\) of degree at most three, let \(g^{[3]}\) denote its homogeneous cubic part.
For a nonnegative integer \(\zeta\), define
\[
 R_\zeta(g,q)
 =\prod_{\substack{p^e\Vert 3b_q\\
 \dim\operatorname{Sing}_{\mathbb F_p}(g^{[3]})>\zeta}}
 p^{e(\dim\operatorname{Sing}_{\mathbb F_p}(g^{[3]})-\zeta)},
\]
as in \cite[Eq.~(3.13)]{bp}.
It records the primes at which the singular locus of \(g^{[3]}\) over \(\mathbb F_p\) has dimension greater than \(\zeta\).

Lemma~\ref{lem:average} leaves a sum over all first shifts.
The proof of \cite[Lemma~3.5]{bp} works with only one class, but we continue the differencing argument uniformly over every class.
For a nonempty set \(E\subset\ZZ^n\setminus\{\boldsymbol 0\}\), allow the later shifts to depend on the first shift:
\[
 \h_i=\h_i(\h_1)\in\ZZ^n,
 \qquad |\h_i|\leq H_i,
\]
for \(\h_1\in E\) and \(2\leq i\leq r\).
For such a choice, write
\[
 G_{\h_1}=F_{\h_1,\ldots,\h_r},
 \qquad
 T_{\h_1}(\alpha;B)=S_{\h_1,\ldots,\h_r}(\alpha;B).
\]

The next lemma makes these choices uniformly on a partition of the first shifts.

\begin{lemma}\label{lem:uniform-selection}
Let \(I\) be a translate of a box of side length \(H_1\) contained in \(|\h|\ll H_1B^\eps\), let \(\alpha\in\RR\), and let \(q\leq B^2\) be a positive integer.
Then \((I\cap\ZZ^n)\setminus\{\boldsymbol 0\}\) can be partitioned into \(O_{d,n,\eps}(B^\eps)\) sets \(E\) with the following property.
For each part \(E\), we can choose an integer \(\zeta\geq\sigma\) and shifts \(\h_i(\h_1)\) as above so that
\[
 \dim\operatorname{Sing}_{\QQ}(G_{\h_1}^{[3]})\leq\zeta,
 \qquad \text{for } \h_1\in E.
\]
Moreover, for every prime \(p\mid3b_q\), the dimension of \(\operatorname{Sing}_{\mathbb F_p}(G_{\h_1}^{[3]})\) is independent of \(\h_1\in E\).
Hence \(R_\zeta(G_{\h_1},q)\) is also independent of \(\h_1\in E\) and we denote its common value by \(R_\zeta\).

With these choices, we have
\begin{equation}\label{eq:uniform-family}
 \sum_{\h_1\in E}
 \frac{|S_{\h_1}(\alpha;B)|}{B^n}
 \ll B^\eps H_1^n
 \left(
 Q^{\sigma-\zeta}R_\zeta^{-1/2}
 \max_{\h_1\in E}
 \frac{|T_{\h_1}(\alpha;B)|}{B^n}
 \right)^{1/2^{r-1}}.
\end{equation}
The implied constant depends only on \(F,\omega\) and \(\eps\).
\end{lemma}

\begin{proof}
Put
\[
 \mathcal P=\{\infty\}\cup\{p:p\mid3b_qd_q\}.
\]
Choose \(C=C(F,d)\geq3d\) so that for every prime \(p>C\),
\[
 \dim\operatorname{Sing}_{\mathbb F_p}(F)=\sigma.
\]
Put
\[
 \mathcal P_C=\{\infty\}\cup\{p\in\mathcal P:p>C\}.
\]
Since all the forms below have degree at most \(d\), the bound \eqref{eq:singular-shift-bound} applies at every place in \(\mathcal P_C\).

The dimension growth bound \eqref{eq:dimension-growth-bound} is uniform for the translated box \(I\).
Indeed, choose \(\mathbf v\in\ZZ^n\) so that every \(\h_1\in I\cap\ZZ^n\) can be written as \(\h_1=\mathbf v+\mathbf y\), with \(|\mathbf y|\ll H_1\).
If \(X_\nu=B_\nu(G,s)\) is one of the varieties from \eqref{eq:singular-shift-bound} controlling \(\h_1\), then
\[
 [\mathbf v+\mathbf y]_\nu\in X_\nu
 \quad\Longleftrightarrow\quad
 [\mathbf y]_\nu\in X_\nu-[\mathbf v]_\nu.
\]
The translated variety \(X_\nu-[\mathbf v]_\nu\) has the same dimension, degree and number of defining equations as \(X_\nu\).
We may therefore apply \eqref{eq:dimension-growth-bound} to \(\mathbf y\) in a box of side length \(O(H_1)\), with a constant independent of the translate \(I\).

At the first stage, partition all shifts \(\h_1\in I\cap\ZZ^n\) according to the tuple
\[
 s_{1,\nu}
 =
 \dim\operatorname{Sing}_{\mathbb F_\nu}
 \bigl(F_{\h_1}^{[d-1]}\bigr),
 \qquad \nu\in\mathcal P.
\]
For \(\h_1\) in a fixed first stage class, let \(\Delta_1\) be the density of that class in \(I\).
Starting from each such \(\h_1\), perform the remaining \(r-1\) differencing steps.
At stage \(i\geq2\), partition the available shifts \(\h_i\) according to the tuple
\[
 s_{i,\nu}
 =
 \dim\operatorname{Sing}_{\mathbb F_\nu}
 \bigl(F_{\h_1,\ldots,\h_i}^{[d-i]}\bigr),
 \qquad \nu\in\mathcal P.
\]
Choose a class with the largest contribution and, within it, a representative maximizing the next differenced sum.
Let \(\Delta_i\) denote the proportion of available shifts in this class.

Refine the first stage partition according to all later tuples \((s_{i,\nu})\), the places selected in the backward dimension growth argument below, and dyadic ranges for \(\Delta_2,\ldots,\Delta_r\).
The shifts \(\h_2,\ldots,\h_r\) may depend on \(\h_1\), while the dimension tuples, selected places, and dyadic ranges are fixed on each resulting part.
This gives the asserted number of parts.
Indeed,
\[
 \#\mathcal P\leq2+\#\{p:p\mid q\},
\]
so, for any \(\eps_0>0\), the number of possible tuples of dimensions at all \(r\) stages is at most
\[
 (n+1)^{r \cdot \#\mathcal P}\ll_{d,n,\eps_0}q^{\eps_0}.
\]
The number of possible sequences of selected places is also \(O_{d,\eps_0}(q^{\eps_0})\).
Finally, every nonzero proportion \(\Delta_i\) satisfies \(\Delta_i\gg_n H_i^{-n}\) and \(\Delta_i\ll_n1\), and \(H_i\leq B\), so there are \(O_n(\log B)\) relevant dyadic ranges at each stage.
At this point choose \(\eps_0>0\) sufficiently small in terms of \(\eps,d,n\).
Since \(q\leq B^2\), the product of the two \(q^{\eps_0}\) losses and the finitely many logarithmic losses is \(O_{d,n,\eps}(B^\eps)\).
Thus the total number of parts is \(O_{d,n,\eps}(B^\eps)\).

Fix one part \(E\).
Let \(\delta_i\) be the upper endpoint of the dyadic range containing \(\Delta_i\), for \(2\leq i\leq r\), and put
\[
 \delta_1=\frac{|E|}{H_1^n}.
\]
Since \(\#(I\cap\ZZ^n)\asymp_n H_1^n\), the set \(E\) is contained in the class defined by its first tuple, and hence \(\delta_1\ll_n\Delta_1\), while \(\Delta_i\leq\delta_i\leq2\Delta_i\) for \(i\geq2\).
Put
\[
 m=2^{r-1},
 \qquad
 A_{\h_1}=\frac{|S_{\h_1}(\alpha;B)|}{B^n},
 \qquad
 U_{\h_1}=\frac{|T_{\h_1}(\alpha;B)|}{B^n}.
\]
For fixed \(\h_1\in E\), and with \(\h_2,\ldots,\h_r\) as selected above, put
\[
 A_i=\frac{|S_{\h_1,\ldots,\h_i}(\alpha;B)|}{B^n},
 \qquad 1\leq i\leq r,
\]
At stage \(i\), there are \(O(B^\eps)\) possible classes.
Choosing a class with the largest contribution and applying Cauchy--Schwarz gives
\[
 A_{i-1}^2\ll B^\eps\delta_iA_i,
 \qquad 2\leq i\leq r.
\]
Since \(r\) is fixed, and \(A_1=A_{\h_1}\) and \(A_r=U_{\h_1}\), iterating this inequality as in \cite[Eq.~(3.16)]{bp} gives
\[
 A_{\h_1}^{m}
 \ll B^\eps
 \prod_{i=2}^{r}\delta_i^{2^{r-i}}
 U_{\h_1}.
\]
H\"older's inequality now yields
\[
 \sum_{\h_1\in E}A_{\h_1}
 \leq |E|^{1-1/m}
 \left(\sum_{\h_1\in E}A_{\h_1}^m\right)^{1/m}
 \ll B^\eps |E|
 \left(
 \prod_{i=2}^{r}\delta_i^{2^{r-i}}
 \max_{\h_1\in E}U_{\h_1}
 \right)^{1/m}.
\]
Since \(|E|=H_1^n\delta_1\), this is
\[
 \sum_{\h_1\in E}\frac{|S_{\h_1}(\alpha;B)|}{B^n}
 \ll B^\eps H_1^n
 \left(
 \prod_{i=1}^{r}\delta_i^{2^{r-i}}
 \max_{\h_1\in E}
 \frac{|T_{\h_1}(\alpha;B)|}{B^n}
 \right)^{1/2^{r-1}}.
\]

It remains to estimate the product of densities, for which we follow the backward calculation in \cite[Proof of Lemma~3.5]{bp}.

Put
\[
 s_{0,\nu}
 =\dim\operatorname{Sing}_{\mathbb F_\nu}(F),
 \qquad s_0=\sigma=\dim\operatorname{Sing}_{\overline{\QQ}}(F).
\]
We now apply \eqref{eq:singular-shift-bound} and \eqref{eq:dimension-growth-bound} backwards.
Set
\[
 t_{r+1}=0,
 \qquad
 \mathcal P_C^{(r)}=\mathcal P_C.
\]
Thus
\[
 \mathcal P_C^{(r)}
 =\{\nu\in\mathcal P_C:s_{r,\nu}\geq t_{r+1}\}.
\]
For \(i=r,r-1,\ldots,1\), choose \(\nu_i\in\mathcal P_C^{(i)}\) indexing a largest term in the corresponding application of \eqref{eq:dimension-growth-bound}, and put
\[
 s_i=s_{i,\nu_i},
 \qquad
 t_i=s_{i-1,\nu_i},
 \qquad
 \mathcal P_C^{(i-1)}
 =\{\nu\in\mathcal P_C:s_{i-1,\nu}\geq t_i\}.
\]
The recursion ensures that
\[
 t_1=\sigma,
 \qquad
 s_i\geq t_{i+1},\quad\text{for }1\leq i<r.
\]
Indeed, the second assertion follows from \(\nu_i\in\mathcal P_C^{(i)}\), and the first follows because \(s_{0,\nu}=\sigma\) for every \(\nu\in\mathcal P_C\).

At stage \(i\), put
\[
 G_{i-1}=
 \begin{cases}
 F,&\text{for }i=1,\\
 F_{\h_1,\ldots,\h_{i-1}}^{[d-i+1]},&\text{for }i\geq2
 \end{cases}.
\]
Then
\[
 F_{\h_1,\ldots,\h_i}^{[d-i]}
 =\h_i\mathbin{\cdot}\nabla G_{i-1}.
\]
For every \(\nu\in\mathcal P_C^{(i)}\), the \(i\)-th shifts selected above therefore lie in \(B_\nu(G_{i-1},s_{i,\nu})\).
By \eqref{eq:singular-shift-bound}, this variety has dimension at most \(n-(s_{i,\nu}-s_{i-1,\nu})_+\).
Apply \eqref{eq:dimension-growth-bound} with \(T=H_i\), divide by the \(\gg H_i^n\) available shifts, choose the largest summand as above, and absorb the number of places into \(B^\eps\).
This gives
\[
 \Delta_i
 \ll B^\eps H_i^{-(s_i-t_i)_+}
 \prod_{\substack{p\in\mathcal P_C^{(i)},\ p<\infty\\
 s_{i,p}-s_{i-1,p}\geq(s_i-t_i)_+}}
 p^{-(s_{i,p}-s_{i-1,p})+(s_i-t_i)_+}.
\]

Define
\[
 \zeta_0=\sigma,
 \qquad
 \zeta_i=\zeta_{i-1}+(s_i-t_i)_+,
 \quad\text{for }1\leq i\leq r.
\]
By the argument in \cite[Proof of Lemma~3.5, following Eq.~(3.20)]{bp}, we have \(s_{i,p}\geq s_{i,\infty}\) for every \(p\in\mathcal P_C\).
Consequently \(s_i\geq s_{i,\infty}\), and an induction gives
\[
 \zeta_i\geq s_i.
\]
Indeed, \(\zeta_0=s_0=\sigma\).
If \(s_i\geq t_i\), then
\[
 \zeta_i\geq s_{i-1}+s_i-t_i\geq s_i;
\]
if \(s_i<t_i\), then \(\zeta_i=\zeta_{i-1}\geq s_{i-1}\geq t_i>s_i\).
Put \(\zeta=\zeta_r\), and define
\[
 L_i
 =
 \prod_{\substack{p\in\mathcal P_C,\ p<\infty\\
 s_{i,p}>\zeta_i}}
 p^{s_{i,p}-\zeta_i}.
\]
Comparing the exponent of each prime in the preceding dimension growth bound gives
\[
 \Delta_i
 \ll B^\eps
 H_i^{\zeta_{i-1}-\zeta_i}\frac{L_{i-1}}{L_i}.
\]
To see this comparison directly, put
\[
 a_p=s_{i,p}-\zeta_i,
 \qquad
 b_p=s_{i-1,p}-\zeta_{i-1}.
\]
If \(p\) occurs in the preceding product, then \(a_p\geq b_p\), and its exponent satisfies
\[
 b_p-a_p\leq(b_p)_+-(a_p)_+.
\]
If \(p\) does not occur, then either \(a_p<b_p\), or \(p\notin\mathcal P_C^{(i)}\).
In the latter case \(s_{i,p}<t_{i+1}\leq\zeta_i\), so \(a_p<0\).
In either case \(0\leq(b_p)_+-(a_p)_+\).
Thus extending the product to every finite \(p\in\mathcal P_C\) and replacing its exponent by \((b_p)_+-(a_p)_+\) can only increase it; the resulting product is \(L_{i-1}/L_i\).

Since \(s_{0,p}=\zeta_0=\sigma\) for \(p>C\), we have \(L_0=1\).
Using \eqref{eq:differencing-scale}, we consequently obtain
\[
 \begin{aligned}
 \Delta_1^{2^{r-1}}\Delta_2^{2^{r-2}}\cdots\Delta_r
 &\ll
 B^\eps
 \prod_{i=1}^r
 H_i^{(\zeta_{i-1}-\zeta_i)2^{r-i}}
 \prod_{i=1}^r
 \left(\frac{L_{i-1}}{L_i}\right)^{2^{r-i}}\\
 &\leq
 B^\eps Q^{\sigma-\zeta}L_r^{-1}.
 \end{aligned}
\]
Here the \(L_i\)-factors telescope in the form
\[
 \prod_{i=1}^r
 \left(\frac{L_{i-1}}{L_i}\right)^{2^{r-i}}
 =
 \frac{L_0^{2^{r-1}}}
 {L_1^{2^{r-2}}\cdots L_{r-1}L_r}
 \leq L_r^{-1}.
\]

The dimensions \(s_{r,p}\) are fixed on \(E\), so \(R_\zeta(G_{\h_1},q)\) is constant on \(E\), as asserted.
Moreover, for \(p>C\) and \(p^e\Vert3b_q\), we have \(e\leq2\).
The contribution of the finitely many primes \(p\leq C\) is bounded in terms of \(F,d,n\).
It follows that
\[
 R_\zeta\ll_{F,d,n}L_r^2,
 \qquad\text{and hence}\qquad
 L_r^{-1}\ll_{F,d,n}R_\zeta^{-1/2}.
\]
Therefore
\[
 \delta_1^{2^{r-1}}\delta_2^{2^{r-2}}\cdots\delta_r
 \ll B^\eps Q^{\sigma-\zeta}R_\zeta^{-1/2}.
\]
Combining this with the preceding application of H\"older's inequality proves \eqref{eq:uniform-family}.
\end{proof}

The selection lemma supplies the uniform family needed for the final cubic estimate, so we can now prove Proposition~\ref{prop:cubic-family}.
\begin{proof}[Proof of Proposition~\ref{prop:cubic-family}]
The zero shift satisfies \(|S_{\boldsymbol 0}(\alpha;B)|\ll B^n\leq B^nH_1^\sigma\), so it remains to estimate the nonzero shifts.
Cover \(0<|\h|\leq H_1B^\eps\) by \(O(B^{n\eps})\) translates of boxes of side \(H_1\).
Fix one covering box \(I\), and then fix one part \(E\) supplied by Lemma~\ref{lem:uniform-selection}.
Let \(\mathrm{Height}(g)\) be the largest absolute value of the coefficients of \(g\), and put
\[
 \mathrm{Height}_{B,3}(g)
 =\mathrm{Height}\bigl(B^{-3}g(B\x)\bigr).
\]
The iterated Cauchy--Schwarz used in the proof of Lemma~\ref{lem:uniform-selection} shows that \(T_{\h_1}(\alpha;B)=0\) implies \(S_{\h_1}(\alpha;B)=0\).
Such shifts make no contribution to \eqref{eq:uniform-family}.
For \(0\leq j\leq3\), a term of degree \(j\) in \(\x\) occurring in \(F_{\h_1,\ldots,\h_r}\), defined by \eqref{eq:iterated-difference}, has total degree \(d-j\) in the shifts, and every one of \(\h_1,\ldots,\h_r\) occurs with positive degree.
Consequently, if \(M=\max_i|\h_i|\), its coefficient is
\[
 \ll_F |\h_1|\cdots|\h_r|M^{3-j}.
\]
If \(T_{\h_1}\neq0\), the common support of the iterated weight gives \(|\h_1|\ll_\omega B\).
We also have \(|\h_1|\leq H_1B^\eps\) and \(|\h_i|\leq H_i\leq B\) for \(i\geq2\).
Hence \(M\ll_\omega B\), and \eqref{eq:differencing-product} shows that multiplying a coefficient of degree \(j\) by \(B^{j-3}\) gives
\begin{equation}\label{eq:final-height}
 \mathrm{Height}_{B,3}
 \bigl(G_{\h_1}\bigr)
 \ll_{F,\omega} B^\eps H_1\cdots H_r
 =B^\eps Q^\tau.
\end{equation}
In view of \eqref{eq:final-height}, choose a constant \(C_{F,\omega}\geq1\) and define the height parameter
\[
 H_0=C_{F,\omega}B^\eps Q^\tau,
\]
large enough that \(\mathrm{Height}_{B,3}(G_{\h_1})\leq H_0\) for every \(\h_1\in E\) with \(T_{\h_1}\neq0\).
Since \(Q\leq B\), this parameter satisfies \(H_0\leq B^{O_{d,\eps}(1)}\).
Moreover, it satisfies the following estimate.
\[
 \min\{c_q,d_qH_0\}
 \ll_{F,\omega}
 B^\eps\min\{c_q,d_qQ^\tau\}.
\]
We also have the following estimate.
\[
 H_0^{1/6}
 \ll_{F,\omega}
 B^{\eps/6}Q^{\tau/6}.
\]
The expression in \cite[Lemma~3.4]{bp} is therefore at most \(B^{O_d(\eps)}K_B(\alpha,q,Q)\).

Suppose first that \(\zeta<n\).
Apply the cubic Poisson estimate \cite[Lemma~3.4]{bp} with its height parameter \(H=H_0\), taking its arbitrary loss to be \(0<\eps_0\leq\eps\).
The uniform support and derivative bounds for the iterated weights verify its weight hypothesis.
We obtain
\[
 \max_{\h_1\in E}
 \frac{|T_{\h_1}(\alpha;B)|}{B^n} \ll
 B^{\eps_0+O_{d,n}(\eps)}R_\zeta^{1/2}
 K_B(\alpha,q,Q)^{n-\zeta}.
\]
The assumed bound on \(K_B(\alpha,q,Q)\) therefore gives
\begin{equation}\label{eq:final-family}
 \max_{\h_1\in E}
 \frac{|T_{\h_1}(\alpha;B)|}{B^n}
 \ll B^{O_{d,n}(\eps)}
 R_\zeta^{1/2}Q^{-(n-\zeta)}.
\end{equation}
If \(\zeta\geq n\), then \(R_\zeta=1\), and the trivial estimate
\[
 \max_{\h_1\in E}\frac{|T_{\h_1}(\alpha;B)|}{B^n}
 \ll_{d,\omega}1
\]
also proves \eqref{eq:final-family}, since \(Q^{\zeta-n}\geq1\).

Substituting \eqref{eq:final-family} into \eqref{eq:uniform-family} cancels the \(R_\zeta\)-factors and gives
\[
 \sum_{\h_1\in E}
 \frac{|S_{\h_1}(\alpha;B)|}{B^n}
 \ll B^{O_{d,n}(\eps)}H_1^n
 Q^{(\sigma-n)/2^{r-1}}
 =B^{O_{d,n}(\eps)}H_1^\sigma.
\]
Summing over the \(O_{d,n,\eps}(B^\eps)\) sets in each box and then over the \(O(B^{n\eps})\) covering boxes gives a total loss \(B^{O_{d,n}(\eps)}\), and proves \eqref{eq:cubic-family}.
\end{proof}

\section{Proof of Theorem~\ref{thm:main}}

We combine the three inputs of Section~\ref{sec:inputs} with the rational covering of Browning and Prendiville.
The averaged estimate leads to \(L_d\), and the pointwise interpolation leads to \(C_d\).
Before beginning the circle method argument, we group the remaining conditions from the pointwise estimates and the major arcs into one auxiliary threshold.

Recall \(A_d=3\cdot2^{d-3}-1\) from Lemma~\ref{lem:moduli}, and put
\[
 \begin{aligned}
 B_k&=(k-1)2^{d-1},
 &&\text{for }3\leq k\leq d,\\
 D_k&=k2^{d-1}+2^{d-k}-2,
 &&\text{for }3\leq k<d.
 \end{aligned}
\]
Direct expansion of \eqref{eq:intro-C} gives
\begin{equation}\label{eq:component-threshold}
 C_{d,k}
 =A_d+\frac dkB_k+\left(\frac dk-1\right)D_k,
 \qquad\text{for }3\leq k<d.
\end{equation}
Put \(\gamma=(3\cdot2^{d-2}-2)^{-1}\).
We next define the exponent \(\xi\) used in the rational covering.
\begin{equation}\label{eq:xi}
 \xi=\frac{5\cdot2^{d-2}-4}{3\cdot2^{d-2}-2}.
\end{equation}
We split the argument according to whether \(\eta_q>B^\gamma\) or \(\eta_q\leq B^\gamma\).
Define
\begin{equation}\label{eq:Ld}
 L_d
 =1+\frac{d-\frac12}{\gamma}
 =\frac34d2^d-2d-\left(3\cdot2^{d-3}-2\right).
\end{equation}
Put
\begin{equation}\label{eq:minor-arc-threshold}
 \Theta_d=\max\{L_d,C_d\}.
\end{equation}

Besides \(L_d\) and \(C_d\), the smooth Hasse argument uses six auxiliary bounds.
The first five make the remaining minor arc contributions \(O(B^{-d-\delta})\), and the sixth is the hypothesis for the major arc asymptotic.
If all six bounds are smaller than \(L_d\), then the condition \(n-\sigma>\Theta_d\) supplies every estimate used in the weighted asymptotic.
Under the nonsingular local solubility hypotheses, the leading constant in that asymptotic is positive, so the weighted count proves the smooth Hasse principle.

To identify the first five bounds, take a reduced approximation \(\alpha=a/q+z\) from \cite[Proof of Lemma~4.2]{bp}.
For such an approximation, we use the following frequency variable.
\begin{equation}\label{eq:minor-frequency}
 u=q|z|=\lVert q\alpha\rVert.
\end{equation}
By the normalized contribution of a dyadic cell, we mean its absolute contribution divided by \(B^n\), with factors \(B^\eps\) and dyadic logarithms suppressed.
Recall the definition of \(\eta_q\) in \eqref{eq:eta}. 
In \cite[Proposition~3.7]{bp}, \(\eta_q^{-1}\) is the height factor that will be kept throughout the downward comparison.
On a dyadic cell with \(\eta_q\leq B^\gamma\), write \(\eta_q\sim R\).
Lemma~\ref{lem:moduli}, summation over the reduced numerators, and integration over the cell give the common factor \(R^{A_d}u\) in its normalized contribution.
With \(\xi\) as in \eqref{eq:xi}, the rational covering also gives \(q\leq B^\xi\) and \(u\leq B^{-\xi}<1/2\).
Thus \(A_d\) replaces the exponent \(7\cdot2^{d-4}-\frac54\) used in \cite[Eq.~(4.2) and Lemma~4.4]{bp}.

We now record the six thresholds \(M_{d,1},\ldots,M_{d,6}\).
\begin{enumerate}
\item The final \(k=3\) term in \cite[Lemma~4.7 and its proof]{bp} is
\[
 \min\left\{R^{-1},u^{1/(5\cdot2^{d-2}-2)}\right\}.
\]
Using \(A_d\) powers of \(R^{-1}\) cancels the modulus factor.
Using a further \((d/\xi-1)(5\cdot2^{d-2}-2)\) powers of the second term leaves \(u^{d/\xi}\leq B^{-d}\).
Define \(M_{d,1}\) by
\begin{equation}\label{eq:Md-cubic}
 M_{d,1}
 =A_d+\left(\frac d\xi-1\right)(5\cdot2^{d-2}-2).
\end{equation}

\item The bound \cite[Eq.~(4.3)]{bp} supplies the two common pointwise factors
\[
 (uB)^{1/(2^{d-1}-2)}
 \qquad\text{and}\qquad
 (qB^{-2})^{1/2^{d-2}}.
\]
The identities \(\xi-1=\gamma(2^{d-1}-2)\) and \(2-\xi=\gamma2^{d-2}\) show that each gives the cell bound \(B^{\gamma A_d-\xi-\gamma(n-\sigma)}\), apart from the factor \(B^\eps\).
We define their common threshold \(M_{d,2}\) by comparison with \(B^{-d}\).
\begin{equation}\label{eq:Md-common}
 M_{d,2}
 =A_d+\frac{d-\xi}{\gamma}.
\end{equation}

\item For \(3\leq k\leq d\), the factor \((qB^{-k})^{1/B_k}\) in \cite[Eq.~(4.4) and Proof of Lemma~4.5]{bp} gives, apart from a factor \(B^\eps\), the cell bound
\[
 B^{\gamma A_d-\xi-(k-\xi)(n-\sigma)/B_k}.
\]
We define \(M_{d,3}\) to be the largest of the thresholds obtained by requiring this bound to be \(O(B^{-d-\delta})\).
\begin{equation}\label{eq:Md-k-dependent}
 M_{d,3}
 =\max_{3\leq k\leq d}
 \frac{B_k(d-\xi+\gamma A_d)}{k-\xi}.
\end{equation}

\item The bound \cite[Eq.~(4.4)]{bp} contains the Weyl factor \(B^{-2^{1-d}}\), which is independent of \(q\) and \(u\).
Its contribution is at most \(B^{-(n-\sigma)/2^{d-1}}\).
We define \(M_{d,4}\) to be the threshold that makes this contribution smaller than \(B^{-d}\).
\begin{equation}\label{eq:Md-Weyl}
 M_{d,4}
 =d2^{d-1}.
\end{equation}

\item The base case \(k=d\) in \cite[Eq.~(4.6) and Proof of Lemma~4.6]{bp} uses the factors \(R^{-1}\) and \((B^du)^{-1/B_d}\).
Using \(A_d\) powers of the first factor and \(B_d\) powers of the second gives
\[
 R^{A_d}uR^{-A_d}(B^du)^{-1}
 =B^{-d}.
\]
Define \(M_{d,5}\) by
\begin{equation}\label{eq:Md-degree-d}
 M_{d,5}
 =A_d+B_d.
\end{equation}

\item The major arc asymptotic in \cite[Lemma~5.1]{bp} assumes \(n-\sigma>\frac34(d-1)2^d\).
We define the major arc threshold \(M_{d,6}\) as follows.
\begin{equation}\label{eq:Md-major-arcs}
 M_{d,6}
 =\frac34(d-1)2^d.
\end{equation}
\end{enumerate}

The conditions from the induction step \(3\leq k<d\) have already been collected in \(C_d\).
We now define \(M_d\) to be the maximum of the six remaining thresholds.
\begin{equation}\label{eq:Md}
 M_d=\max_{1\leq i\leq6}M_{d,i}.
\end{equation}

\begin{lemma}\label{lem:threshold-absorption}
For every \(d\geq5\), the quantity \(M_d\) defined by \eqref{eq:Md} satisfies \(M_d<L_d\).
\end{lemma}

\begin{proof}
The definitions \eqref{eq:Ld} and \eqref{eq:Md-common} give \(\gamma A_d=1/2\) and
\[
 L_d-M_{d,2}
 =2^{d-1}-1>0.
\]
Using \eqref{eq:Md-cubic}, direct simplification gives
\[
 L_d-M_{d,1}
 =\frac{10\cdot2^{2d-4}-(6d+3)2^{d-2}+4d-4}
 {5\cdot2^{d-2}-4}.
\]
Since \(2^{d-2}\geq2d-2\), the numerator is at least \(2^{d-2}(14d-23)+4d-4>0\).

For \(3\leq k\leq d\), we have
\[
 \frac{B_k(d-\xi+\gamma A_d)}{2^d(k-\xi)}
 =\frac{d+\frac12-\xi}{2} \cdot \frac{k-1}{k-\xi}.
\]
By \eqref{eq:xi} we know \(1<\xi<2\), so the last factor decreases with \(k\), and hence
\[
 \frac{B_k(d-\xi+\gamma A_d)}{2^d(k-\xi)}
 \leq\frac{3d}{4}-\frac78
 -\frac{2d-5}{8(2^{d-1}-1)}
 <\frac{3d}{4}-\frac78
 <\frac{L_d}{2^d},
\]
where the last inequality follows from
\[
 \frac{L_d}{2^d}
 -\left(\frac{3d}{4}-\frac78\right)
 =\frac12-\frac{d-1}{2^{d-1}}>0.
\]
Taking the maximum over \(k\) as in \eqref{eq:Md-k-dependent} gives \(M_{d,3}<L_d\).
Finally, the definitions \eqref{eq:Md-Weyl}, \eqref{eq:Md-degree-d}, and \eqref{eq:Md-major-arcs} give
\begin{align*}
 M_{d,4}-M_{d,5}
 &=2^{d-3}+1>0,\\
 M_{d,6}-M_{d,4}
 &=(d-3)2^{d-2}>0,\\
 L_d-M_{d,6}
 &=\frac{3\cdot2^{d-2}-4d+4}{2}>0.
\end{align*}
Thus \(M_{d,i}<L_d\) for \(1\leq i\leq6\), which proves the lemma.
\end{proof}

We now follow the minor arc decomposition in \cite[Section~4]{bp} and split the rational covering according to \(\eta_q\).
For \(\eta_q>B^\gamma\), we combine Lemma~\ref{lem:average} with Proposition~\ref{prop:cubic-family}; for \(\eta_q\leq B^\gamma\), we use their approach with Lemma~\ref{lem:moduli}.
Fix \(\Delta=1/6\).
With \(\xi\) as in \eqref{eq:xi} and \(u\) as in \eqref{eq:minor-frequency}, after the range of tiny frequencies in the first part of the rational covering in \cite[Proof of Lemma~4.2]{bp} is removed, the remaining arcs are contained in the union of the following two ranges.
\begin{equation}\label{eq:minor-range-first}
 (q,u)\in[B^\Delta,B^\xi]\times[B^{-d-2},B^{-\xi}],
\end{equation}
and
\begin{equation}\label{eq:minor-range-second}
 (q,u)\in[1,B^\xi]\times[B^{-d+\Delta},B^{-\xi}].
\end{equation}
For fixed \(q\) and a dyadic scale \(u_0\), the absolute contribution of the cell \(u\sim u_0\) is
\[
 \sum_{\substack{a\bmod q\\(a,q)=1}}
 \int_{u_0/q\leq|z|<2u_0/q}
 |S_\omega(a/q+z;B)|\dint z,
\]
with the final annulus truncated at \(q|z|\leq B^{-\xi}\).

For a fixed weight \(\omega\) as in Lemma~\ref{lem:average}, define \(E_{\mathrm{large}}(B)\) by summing the displayed absolute contributions of \(S_\omega\) over all pairs \((q,u_0)\) lying in one of the ranges \eqref{eq:minor-range-first} or \eqref{eq:minor-range-second}, and satisfying \(\eta_q>B^\gamma\).
Define \(E_{\mathrm{small}}(B)\) in the same way, with the restriction \(\eta_q\leq B^\gamma\).

We first estimate the contribution from \(\eta_q>B^\gamma\).
\begin{lemma}\label{lem:large-height}
If \(n-\sigma>L_d\), then \(E_{\mathrm{large}}(B)\ll B^{n-d-\delta}\) for some \(\delta>0\).
\end{lemma}

\begin{proof}
We use the following two scales.
\[
 Q_0=B^{2^{d-3}\gamma},\qquad
 H_1=Q_0^{1/2^{d-4}}=B^{2\gamma}.
\]
Since \(d\geq5\), both exponents \(2^{d-3}\gamma\) and \(2\gamma\) lie in \((0,1)\), so \(1\leq Q_0,H_1\leq B\).
Equation~\eqref{eq:xi} gives \(1<\xi<2\), and hence
\[
 \kappa=d-1+2\gamma-2\xi>d-5+2\gamma>0.
\]
Put \(\mu=\gamma(n-\sigma-L_d)>0\).
Fix a constant \(C = C(d, n)>0\), so that each total exponent denoted by \(O_{d,n}(\eps)\) below is at most \(C\eps\), and choose
\[
 0<\eps<\min\{\kappa,2-\xi,\mu/(2C)\}.
\]
Let \(u\sim u_0\), where \(u_0\) runs over dyadic frequency scales.
Apply Lemma~\ref{lem:average} with \(H=H_1\) and
\[
 t=\frac{u_0}{q},\qquad
 \lambda=(H_1B^{d-1})^{-1}.
\]
For \(z\) in the enlarged annulus in \eqref{eq:shift-average}, put \(u'=q|z|\).
Since
\[
 qB^\eps\lambda
 \ll B^{\xi+\eps-(d-1)-2\gamma}
 \ll B^{-\xi},
\]
the same reduced numerator still realizes the distance to the nearest integer, and \(u'\ll B^{-\xi}\).
Since \(q\leq B^\xi\leq B^{2-\eps}\) and \(u'\leq B^{-1}\) for sufficiently large \(B\), all the remaining size hypotheses of Proposition~\ref{prop:cubic-family} are satisfied.
It remains to check the condition involving \(K_B\) in that proposition.
By \eqref{eq:cubic-bound}, and since \(\|q(a/q+z)\|=u'\), we have
\[
\frac{\sqrt q}{B}
+\sqrt{u'BQ_0^\tau}
 + \frac{Q_0^{\tau/6}\min\{c_q,d_qQ_0^\tau\}^{1/6}}{b_q^{1/2}(c_qd_q)^{1/3}}.
\]
Equation~\eqref{eq:xi}, the definition of \(\gamma\), and \eqref{eq:tau} give the identities
\[
 1-\frac\xi2
 =\frac{\xi-1}{2+\tau}
 =2^{d-3}\gamma.
\]
These bound the first two terms in the display above by \(O(Q_0^{-1})\).

Raising the third term to the sixth power gives
\begin{equation}\label{eq:large-height-third-term}
 \min\left\{
 \frac{Q_0^\tau}{b_q^3c_qd_q^2},
 \frac{Q_0^{2\tau}}{b_q^3c_q^2d_q}
 \right\}.
\end{equation}
Equation~\eqref{eq:tau} also gives the following identities.
\[
 (6+\tau)2^{d-3}=2^d-2,\qquad
 (6+2\tau)2^{d-3}=5\cdot2^{d-2}-4.
\]
Together with the definition of \(\eta_q\), these identities show that \eqref{eq:large-height-third-term} is \(O(Q_0^{-6})\) when \(\eta_q>B^\gamma\).
Since \(\eps\) is fixed, the sum of these three bounds is at most \(B^\eps Q_0^{-1}\) for all sufficiently large \(B\).
Proposition~\ref{prop:cubic-family} therefore applies throughout the enlarged annulus.
It gives the following uniform estimate.
\[
 \Sigma_{H_1}
 \ll B^{n+O_{d,n}(\eps)}H_1^\sigma.
\]
The main term in Lemma~\ref{lem:average} consequently has the following form.
\[
 (t+\lambda)
 B^{n-\frac12+O_{d,n}(\eps)}
 H_1^{(\sigma-n+1)/2}
 =
 (t+\lambda)
 B^{n-\frac12-\gamma(n-\sigma-1)+O_{d,n}(\eps)}.
\]
After summing over the reduced numerators, the two terms \(t=u_0/q\) and \(\lambda\) contribute \(u_0\) and \(q\lambda\), respectively.
Let \(U(q)\) be the set of dyadic scales \(u_0\) for which \((q,u_0)\) lies in one of the ranges \eqref{eq:minor-range-first} or \eqref{eq:minor-range-second}.
Then
\[
 \#U(q)\ll\log B,
 \qquad
 \sum_{u_0\in U(q)}u_0\ll B^{-\xi}.
\]
Thus the first term satisfies the explicitly dyadic sum
\[
 \sum_{q\leq B^\xi}
 \sum_{u_0\in U(q)}u_0
 \ll B^\xi B^{-\xi}\ll1.
\]
For the second term, summing over the dyadic set \(U(q)\) gives
\[
 \lambda
 \sum_{q\leq B^\xi}
 \sum_{u_0\in U(q)}q
 \ll
 (\log B)\lambda
 \sum_{q\leq B^\xi}q
 \ll
 B^{-(d-1+2\gamma-2\xi)+\eps}.
\]
The \(B^{-J}\) term in Lemma~\ref{lem:average} contributes at most \(B^{2\xi-J+\eps}\) after these sums and is negligible when \(J\) is sufficiently large.
It follows that the total contribution from \(\eta_q>B^\gamma\) satisfies the estimate below.
\begin{equation*}
 E_{\mathrm{large}}(B)
 \ll B^{n-1/2-\gamma(n-\sigma-1)+O_{d,n}(\eps)}
 =B^{n-d-\mu+O_{d,n}(\eps)}
 \ll B^{n-d-\mu/2}.
\end{equation*}
This is the asserted fixed saving.
\end{proof}

We next estimate the contribution from \(\eta_q\leq B^\gamma\).
\begin{lemma}\label{lem:small-height}
If \(n-\sigma>\Theta_d\), then \(E_{\mathrm{small}}(B)\ll B^{n-d-\delta}\) for some \(\delta>0\).
\end{lemma}

\begin{proof}
Choose \(\eps>0\) sufficiently small in terms of \(d,n\) and the fixed gap \(n-\sigma-\Theta_d\).
Write \(\eta_q\sim R\leq B^\gamma\) and \(u\sim u_0\).
Lemma~\ref{lem:moduli} gives \(O(B^\eps R^{A_d})\) possible moduli with \(\eta_q\sim R\), after absorbing its logarithmic factor.
For each such \(q\), there are at most \(q\) reduced numerators, while the corresponding \(z\)-annulus has length \(O(u_0/q)\); their product is \(O(u_0)\).
Thus, if \(X(q,u)\) is any pointwise factor under consideration, its normalized contribution from this dyadic cell is
\[
 \ll B^\eps R^{A_d}u_0
 \sup_{\substack{\eta_q\sim R\\u\sim u_0}}
 X(q,u)^{n-\sigma}.
\]

Apply the pointwise estimates \cite[Propositions~3.6 and~3.7]{bp} to the expressions in \cite[Eq.~(4.3) and Eq.~(4.4)]{bp}.
If \(X_1,\ldots,X_m\) are the finitely many nonnegative factors in one of these expressions, then
\[
 (X_1+\cdots+X_m)^{n-\sigma}
 \leq m^{n-\sigma-1}\sum_{j=1}^m X_j^{n-\sigma}.
\]
It therefore suffices to estimate the contribution of each factor separately.
By the definitions of \(M_{d,2}\), \(M_{d,3}\), and \(M_{d,4}\), each factor covered by \eqref{eq:Md-common}, \eqref{eq:Md-k-dependent}, or \eqref{eq:Md-Weyl} contributes \(O(B^{n-d-\delta})\) to the present dyadic cell for some \(\delta>0\), whenever \(n-\sigma>\max\{M_{d,2},M_{d,3},M_{d,4}\}\).

For the remaining terms, we keep \(\eta_q^{-1}\) from \cite[Eq.~(4.3)]{bp} throughout the downward comparison in \cite[Eq.~(4.6) and Proofs of Lemmas~4.6 and~4.7]{bp}.
Put
\[
 Y_k=u^{1/(B_k+2^{d+1-k}-2)},
 \qquad
 Z_k=(B^ku)^{-1/B_k}.
\]
After separating the terms covered by \eqref{eq:Md-common}, \eqref{eq:Md-k-dependent}, and \eqref{eq:Md-Weyl}, \cite[Eq.~(4.3)]{bp} contributes \(\eta_q^{-1}\), while \cite[Eq.~(4.4)]{bp} contributes \(Y_k+Z_k\) for each \(k\), since its final minimum is at most \(Z_k\).
Taking the minimum over these alternatives and using \(\eta_q^{-1}\asymp R^{-1}\), the remaining pointwise factor is, up to a constant depending only on \(d\), bounded by
\begin{equation}\label{eq:downward-factor}
 \min\left\{R^{-1},\min_{3\leq k\leq d}(Y_k+Z_k)\right\}.
\end{equation}

Put \(A=R^{-1}\).
Repeatedly splitting \(Y_k+Z_k\) in \eqref{eq:downward-factor}, from \(k=d\) down to \(k=3\), gives
\begin{equation}\label{eq:downward-decomposition}
 \begin{aligned}
 &\min\left\{A,\min_{3\leq k\leq d}(Y_k+Z_k)\right\}\\
 &\quad\leq \min\{A,Z_d\}
 +\min\left\{A,Y_d,\min_{3\leq k<d}(Y_k+Z_k)\right\}\\
 &\quad\leq \min\{A,Z_d\}
 +\sum_{k=3}^{d-1}\min\{A,Y_{k+1},Z_k\}
 +\min\{A,Y_3\}.
 \end{aligned}
\end{equation}
At level \(k\), \(Y_{k+1}\) is an entry of the remaining minimum.
After splitting \(Y_k+Z_k\), the assumption \(0<u\leq1\) gives \(Y_k\leq Y_{k+1}\) and we replace \(Y_{k+1}\) by \(Y_k\) at the next level.
Thus \eqref{eq:downward-decomposition} reduces \eqref{eq:downward-factor} to the base term \(\min\{A,Z_d\}\), the terms \(\min\{A,Y_{k+1},Z_k\}\) for \(3\leq k<d\), and the final term \(\min\{A,Y_3\}\).
For each such factor \(X\), it remains to bound the normalized cell contribution
\[
 B^\eps R^{A_d}uX^{n-\sigma}.
\]

For the base case \(k=d\), Lemma~\ref{lem:threshold-absorption} gives \(n-\sigma>M_{d,5}=A_d+B_d\).
Since \(Z_d^{B_d}=(B^du)^{-1}\) and a minimum is no larger than either of its entries,
\[
 R^{A_d}u
 \min\{R^{-1},Z_d\}^{n-\sigma}
 \leq
 B^{-d}
 \min\{R^{-1},Z_d\}^{n-\sigma-M_{d,5}}.
\]
On the range \eqref{eq:minor-range-first}, \eqref{eq:eta} and \(q\geq B^\Delta\) give \(R\gg B^{\Delta/(5\cdot2^{d-2}-4)}\), so the remaining minimum gives a fixed saving through its first entry.
On the range \eqref{eq:minor-range-second}, we have \(B^du\geq B^\Delta\), so its second entry gives a fixed saving.

For the final term \(\min\{R^{-1},Y_3\}\), Lemma~\ref{lem:threshold-absorption} gives \(n-\sigma>M_{d,1}\).
The definition \eqref{eq:Md-cubic} similarly gives
\[
R^{A_d}u
 \min\{R^{-1},Y_3\}^{n-\sigma}\leq
 u^{d/\xi}
 \min\{R^{-1},Y_3\}^{n-\sigma-M_{d,1}}.
\]
Here \(u^{d/\xi}\leq B^{-d}\), and the second entry of the remaining minimum gives a fixed saving.

In the induction step \(3\leq k<d\), use \(A_d\), \(dB_k/k\), and \((d/k-1)D_k\) powers of the three entries in the corresponding minimum.
Using \eqref{eq:component-threshold} and assigning the remaining powers to \(Y_{k+1}=u^{1/D_k}\), we obtain
\[
 \begin{aligned}
 R^{A_d}u
 \min\{R^{-1},Z_k,Y_{k+1}\}^{n-\sigma}
 & \leq B^{-d} \cdot u^{(n-\sigma-C_{d,k})/D_k}\\
 &
 \leq B^{-d-\xi(n-\sigma-C_{d,k})/D_k}.
 \end{aligned}
\]
Finally, \(n-\sigma>\Theta_d\) handles every induction step by \eqref{eq:component-threshold}, and \(n-\sigma>L_d>M_d\) handles the base case, the final \(k=3\) term, and all remaining pointwise terms by Lemma~\ref{lem:threshold-absorption}.
Every inequality is strict, so the resulting fixed margins absorb the dyadic logarithms and the \(B^\eps\) losses.
This proves the lemma.
\end{proof}

The preceding two estimates now control all the remaining minor arcs, so we can complete the proof of the main theorem.
\medskip

\begin{proof}[Proof of Theorem~\ref{thm:main}]
For \(d=5,6\), equation \eqref{eq:Ld} gives \(L_5=100\) and \(L_6=254\).
In particular, \(L_d\in\ZZ\), so the condition \(n-\sigma>L_d\) is equivalent to \(n-\sigma\geq L_d+1\), which is precisely the first lower bound in Theorem~\ref{thm:main}.
Direct substitution also gives \(L_d>C_d\).
For \(d\geq7\), we have \(C_d\geq C_{d,3}\) and
\[
 C_{d,3}-L_d
 =(d-7)2^{d-3}+\frac{4d}{3}-1>0.
\]
Thus the hypothesis of Theorem~\ref{thm:main} implies
\[
 n-\sigma>\Theta_d,
\]
in every degree \(d\geq5\).
Let \(\x_0\) and \(\omega\) be as in Section~\ref{sec:averaging}.
Under the condition \(n-\sigma>\Theta_d\), with \(\Theta_d\) defined in \eqref{eq:minor-arc-threshold}, we shall prove the following weighted asymptotic for some \(\delta>0\):
\begin{equation}\label{eq:weighted-asymptotic}
 \sum_{\substack{\x\in\ZZ^n\\F(\x)=0}}\omega(\x/B)
 =\mathfrak S\mathfrak I B^{n-d}
 +O(B^{n-d-\delta}).
\end{equation}
Here \(\mathfrak S\) and \(\mathfrak I\) are the singular series and singular integral.
We use the rational covering, pointwise estimates, and major arcs from \cite[Sections~3--5]{bp}.
In the first covering range, the trivial estimate removes \(u<B^{-d-2}\) with total contribution \(O(B^{n-d-(2-\xi)+\eps})\), exactly as in \cite[Proof of Lemma~4.2]{bp}.
By~\eqref{eq:xi} we know \(\xi<2\), so this is a power saving when \(\eps\) is sufficiently small.
Lemmas~\ref{lem:large-height} and~\ref{lem:small-height} give a power saving on all the remaining minor arcs.
Moreover, Lemma~\ref{lem:threshold-absorption} gives
\[
 n-\sigma>M_{d,6}
 =\frac34(d-1)2^d,
\]
which is the hypothesis required for the major arcs by \cite[Lemma~5.1]{bp}.
The major arc formula now gives \eqref{eq:weighted-asymptotic}.
The conditions \(\omega(\x_0)>0\) and \(\partial_1F\neq0\) on its support imply \(\mathfrak I>0\) by the implicit function theorem.
Nonsingular \(p\)-adic solubility implies \(\mathfrak S>0\), as in the discussion following \cite[Lemma~5.1]{bp}.
For a sufficiently large integer \(B\), the weighted count is positive and therefore supplies \(\x\in\ZZ^n\) with \(F(\x)=0\) and \(\x/B\in\operatorname{supp}(\omega)\).
By homogeneity, \(\x/B\) is a rational zero of \(F\) in \(\operatorname{supp}(\omega)\), and it is nonsingular by the choice of support.
\end{proof}

We now estimate the bound \(C_d\) as \(d\) grows.
\medskip

\begin{proof}[Proof of Corollary~\ref{cor:asymptotic}]
For \(3\leq k<d\), substitution in \eqref{eq:intro-C} gives
\[
 \frac{C_{d,k}}{2^d}=d+\frac38-\phi_d(k),
\]
where
\[
 \phi_d(k)=\frac12\left(k+\frac dk\right)
 -\left(\frac dk-1\right)(2^{-k}-2^{1-d})
 +2^{-d}.
\]
Consequently,
\[
 \frac{C_d}{2^d}
 =d+\frac38-\min_{3\leq k<d}\phi_d(k).
\]

Put \(y=\sqrt d/k\).
Since \(k<d\), the terms \((d/k-1)2^{1-d}\) and \(2^{-d}\) are positive.
Discarding these terms and completing the square gives
\[
 \begin{aligned}
 \phi_d(k)-\sqrt d
 &>k\left(\frac12-2^{-k}\right)y^2-ky
   +\frac k2+2^{-k}\\
 &\geq2^{-k}-\frac{k}{2^k-2}>-\frac12.
 \end{aligned}
\]
Here the last inequality follows from \(2^k-2\geq2k\) for \(k\geq3\).
For the reverse bound, let \(k\) be the integer nearest to \(\sqrt d\).
For \(d\geq7\), this choice satisfies \(3\leq k<d\) and \(k=\sqrt d+O(1)\).
For this choice, the definition of \(\phi_d\) gives
\[
 \phi_d(k)\leq\frac12\left(k+\frac dk\right)+2^{-d}
 =\sqrt d+O(d^{-1/2}).
\]
Thus
\[
 \min_{3\leq k<d}\phi_d(k)
 =\sqrt d+O(1),
\]
and hence
\[
 C_d=
 \left(d-\sqrt d+O(1)\right)2^d.
\]
The same lower bound for \(\phi_d(k)\) gives, for every \(d\geq7\),
\[
 \frac{C_d}{2^d}
 <d-\sqrt d+1,
\]
which proves the uniform estimate.
\end{proof}

\bibliographystyle{amsalphaauthorsort}
\bibliography{refs}

\end{document}